\documentclass{amsart}
\usepackage[leqno]{amsmath}
\usepackage{amsfonts}
\usepackage{amsthm}
\usepackage{amsmath}
\usepackage{amssymb}
\usepackage{amscd}
\usepackage{inputenc}
\usepackage{t1enc}
\usepackage[mathscr]{eucal}
\usepackage{indentfirst}
\usepackage{graphicx}
\usepackage{graphics}
\numberwithin{equation}{section}
\usepackage{epstopdf} 
\usepackage{multirow}
\usepackage{mathtools}
\usepackage{caption}
\usepackage{setspace}
\usepackage{tikz-cd}
\usepackage{multicol}
\usepackage{comment}
\usepackage[title,titletoc]{appendix}

\usepackage[style=alphabetic, backend=biber, url = false, isbn = false, doi = false]{biblatex}
\usepackage[cache=false]{minted}
 
\newcommand{\bibhack}[1]{}

\theoremstyle{plain}
\newtheorem*{theorem*}{Theorem}
\newtheorem*{lemma*}{Lemma}
\newtheorem*{cor*}{Corollary}
\newtheorem*{lem*}{Lemma}
\newtheorem{Th}{Theorem}[section]

\newtheorem{lemma}[Th]{Lemma}

\newtheorem{cor}[Th]{Corollary}
\newtheorem{Prop}[Th]{Proposition}

 \theoremstyle{definition}

\newtheorem*{Def*}{Definition}

\newtheorem{?}[Th]{Problem}

\theoremstyle{plain}
\newtheorem{innercustomgeneric}{\customgenericname}
\providecommand{\customgenericname}{}
\newcommand{\newcustomtheorem}[2]{%
  \newenvironment{#1}[1]
  {%
   \renewcommand\customgenericname{#2}%
   \renewcommand\theinnercustomgeneric{##1}%
   \innercustomgeneric
  }
  {\endinnercustomgeneric}
}

\newcustomtheorem{customthm}{Theorem}
\newcustomtheorem{customtheorem}{Theorem}
\newcustomtheorem{customlemma}{Lemma}
\newcustomtheorem{customcor}{Corollary}
\newcustomtheorem{customconj}{Conjecture}

\newcommand{\codim}{\operatorname{codim}}

\newcommand{\diam}{{\rm{diam}}}

\newcommand{\setvert}{\,\vert \,}

\newcommand{\Q}{\mathbb{Q}}
\newcommand{\Z}{\mathbb{Z}}
\newcommand{\R}{\mathbb{R}}

\newcommand{\into}{\xhookrightarrow{}}

\newcommand{\OO}{\mathcal{O}}

\newcommand{\norm}[1]{\left\lVert#1\right\rVert}

\newcommand{\bb}{\mathbb}

\newcommand{\orb}{{\rm orb}}

\newcommand{\SO}{{\rm SO}}

\newcommand{\homeo}{{\rm homeo}}
\newcommand{\diag}{{\rm diag}}

\newcommand{\Id}{{\rm Id}}
\newcommand{\Curv}{{\rm Curv}}

\title{No Connectedness Lemma for Alexandrov spaces}
\author[Nienhaus]{Jan Nienhaus}
\address{KU Leuven, Celestijnenlaan 200B, Heverlee, Belgium}
\email{\detokenize{jan.nienhaus@kuleuven.be}}
\author[Wulle]{Dennis Wulle}
\address{Universit\"at M\"unster, Einsteinstra\ss e 62, M\"unster, Germany}
\email{\detokenize{dennis.wulle@uni-muenster.de}}
\date{September 2026}

\begin{document}

\begin{abstract}
    Since its introduction in 2003, Wilking's Connectedness Lemma has become a powerful tool in the study of manifolds with positive sectional curvature. In recent work of the authors, this has been extended to orbifolds. The purpose of this note is to exhibit counterexamples in the class of Alexandrov spaces for the Connectedness Lemma in both its ordinary and dual formulations. This answers a 2007 question of Wilking.
\end{abstract}
\maketitle
\section{Introduction}
One of the fundamental underlying questions in global geometry is this: \textit{How does the geometry of a space restrict its topology}? The resulting theorems are interesting because of their necessary local-to-global nature: the common geometric restrictions, i.e. curvature properties, are local; the topological conclusions are not.

In the context of positive sectional curvature, there has been great progress over the last 25 years \cite{Wil06, gwz, Kennard13, Kennard2021-ge,Nienhaus22, Wulle2023, Kennard2026-rk,NienhausWulle2026t3}, all using the same tool to bring the curvature restriction into the argument: Wilking's Connectedness Lemma \cite{Wil03}:

\begin{theorem*}[Connectedness Lemma \cite{Wil03}]
    Let $M^n$ be a closed Riemannian manifold with positive sectional curvature and $N^{n-k} \subset M$ a closed totally geodesic submanifold. Then the inclusion $\iota \colon N \hookrightarrow M$ is  $(n - 2k+1)$-connected. 
\end{theorem*}

For a map to be $k$-connected means that the induced map on homotopy groups is an isomorphism in degrees $< k$ and surjective in degree $k$. This puts the conclusion firmly into the topological realm and can quickly lead to strong restrictions on the cohomology, at least in settings with Poincar\'e duality \cite{Wil03}.

The powerful consequences this has led to raise an obvious question: Since we can extend the notion of positive sectional curvature to more general spaces than manifolds, is that assumption critical for the Connectedness Lemma? If yes, is there perhaps a different formulation that does go through, as is often the case when passing to singular spaces? Here there is an obvious second candidate:

\begin{theorem*}[Connectedness Lemma, dual formulation]
    Let $M^n$ be a closed manifold with positive sectional curvature. Let $N^{n-k}$ be a closed totally geodesic submanifold. Then $H_i(M \setminus N;R) = 0$ for $i \ge 2k-1$ for any coefficient ring $R$.
\end{theorem*}

For manifolds, this follows from the usual formulation as a direct consequence of Poincar\'e duality and excision, and is equivalent if everything is simply connected. Where Poincar\'e duality is unavailable, it is not clear which version is the more likely to extend.

There are two canonical candidate classes when dealing with positively curved singular spaces: Orbifolds, which allow only very mild singularities, and Alexandrov spaces, which may be regarded as the most general class currently considered where positive curvature still makes sense.

It turns out that for orbifolds, a version of the Connectedness Lemma does hold \cite{NienhausWulle2026t3}:

\begin{lemma*}
    Let $\OO^n$ be a closed Riemannian orbifold with positive sectional curvature and $N^{n-k} \subset\OO$ a closed totally geodesic suborbifold. Then the inclusion $\iota \colon N \hookrightarrow \OO$ is orbifold-$(n - 2k+1)$-connected. 
\end{lemma*}
Here slight care has to be taken with the conclusion. It is not true that the induced map on the underlying spaces has to be highly connected. Instead, $f$ is said to be orbifold-$(n-2k+1)$-connected if the induced map on \textit{classifying} spaces $Bf:BN\into B\OO$ is $(n-2k+1)$-connected. See \cite{NienhausWulle2026t3} for details.

The purpose of this note is:
\begin{theorem*}
    For both formulations, there are Alexandrov space counterexamples $N\subset A$ with $Curv\ge1$ in the sense of Alexandrov. One may even arrange that $N$ is a totally convex (not only totally geodesic) subspace of $A$. The counterexamples work for any choice of coefficient ring.
\end{theorem*}
Given such an example, the spherical joins $X\star N\subset X\star A$ give new counterexamples for any Alexandrov space $X$, giving counterexamples in many dimensions and with arbitrarily bad local geometry.

That the usual formulation fails has been known for over 20 years; the main new result is that the dual formulation fails too. This answers a question raised by Wilking in 2007\footnote{At the 2007 AIM Workshop "On Manifolds with nonnegative sectional curvature"}.

If $N$ has codimension $1$ in $A$, then the dual Connectedness Lemma is true in any coefficient field (Lemma \ref{codim1}). The counterexamples we describe are of codimension $3$, but examples can be obtained for codimension $2$ as well. The full conclusion of the ordinary formulation is known to fail in any codimension. However, the statement only for $\pi_0$, i.e. Frankel's theorem, does hold for Alexandrov spaces \cite[Theorem 2.3]{Petrunin1998}.

The paper is structured as follows: In a preliminary section we will establish the dual formulation for manifolds with arbitrary coefficients and orbifolds with rational coefficients. While this was known, there is no proof in the literature, so we include it. Then in Section \ref{sec:counter} we exhibit the counterexamples. For the standard version one can simply take suspensions; this has been known and is included only for the sake of completeness. The counterexamples for the dual formulation are given by quotients of spheres by finite groups, with the main difficulty being to recognize the topology of the complement. The subspace involved is totally convex, ruling this out as a saving additional assumption.

It should be noted that although both the ambient space and subspace in the example are orbifolds, this is not a counterexample to the orbifold Connectedness Lemma of \cite{NienhausWulle2026t3}. Part of the subtleties with orbifolds is that a totally geodesic metric subset of an orbifold is not necessarily a suborbifold. In this case the subset $N$ does \textit{not} define a suborbifold of $A$ in the sense of \cite{NienhausWulle2026t3}. For more on this see the discussion there.

\subsection*{Acknowledgements}
The authors would like to thank Burkhard Wilking for bringing the problem to their attention, and Alexander Lytchak for pointing out that the dual version is true for a codimension $1$ subset.

The first named author would like to thank the Cluster of Mathematics M\"unster for enabling a research stay in M\"unster in the summer of 2026, where part of this work was carried out.

The first named author is supported by the FWO and FNRS under EOS project G0I2222N.
The second named author acknowledges support by the Alexander von Humboldt Foundation through Gustav Holzegel’s Alexander von Humboldt Professorship endowed by the Federal Ministry of Education and Research and by the Deutsche Forschungsgemeinschaft (DFG, German Research Foundation) under Germany's Excellence Strategy EXC 2044/2 –390685587, Mathematics Münster: Dynamics–Geometry–Structure.

\subsection*{AI disclosure}
The explicit counterexamples were first found by OpenAI's Chat\-GPT-6-Astra with minimal human input. 

The structuring of the exposition, including all statements and proofs as appearing in this article, was done by the authors. This document contains no machine-generated text.

\section{Preliminaries}\label{prelim}

We start with the proof of the dual statement to the Connectedness Lemma. The statement is known to the experts, but there is no proof in the literature yet.

\begin{cor}\mbox{}
    \begin{enumerate}
        \item Let $M^n$ be a closed Riemannian manifold with positive sectional curvature and $N^{n-k} \subset M$ a closed totally geodesic submanifold. Then $H_i(M\setminus N; R) = 0$ for $i \ge 2k - 1$ and any coefficient ring $R$.
        \item Let $\OO^n$ be a closed Riemannian orbifold with positive sectional curvature and $N^{n-k} \subset \OO$ a closed totally geodesic suborbifold. Then $H_i(\vert \OO \vert \setminus \vert N \vert ;\Q) = 0$ for $i \ge 2k -1$.
    \end{enumerate}
\end{cor}
    \begin{proof}\mbox{}
   We start with the first statement. By the Connectedness Lemma, we have $H^i(M,N;R) = 0$ for $i = 0, \ldots, n-2k+1$. 
        
        We first assume that $M$ is orientable. By Poincar\'e-Lefschetz duality and excision $H^{n-i}(M,N;R) = H_i(M\setminus N;R)$ follows and hence $H_i(M\setminus N;R) = 0$ for $i \ge 2k -1$. 

        Now let $M$ be non-orientable. For $R = \Z_2$, the statement follows as before. Consider a general coefficient ring $R$. Define the $\Z_2$-vector spaces $$R[2] = \lbrace a \in R \setvert 2a = 0 \rbrace \quad  R_2 = R/2R. $$
        Since the statement holds true for $R= \Z_2$, we also have 
        $$H_i (M \setminus N; R[2]) = H_i(M\setminus N; R_2) = 0 \text{ for } i \ge 2k - 1.$$
        The short exact sequences 
        \begin{align*}
            0 \to R[2] \to R \to 2R \to 0\\
            0 \to 2R \to R \to R_2 \to 0
        \end{align*}
        give long exact sequences for the homology groups 
        \begin{align*}
             \to H_i(M\setminus N; R[2]) \to  H_i(M\setminus N; R) \to  H_i(M\setminus N; 2R) \to H_{i-1}(M\setminus N; R[2]) \to 
             \end{align*}
             and
             \begin{align*}
              \to H_i(M\setminus N; 2R) \to  H_i(M\setminus N; R) \to  H_i(M\setminus N; R_2) \to H_{i-1}(M\setminus N; 2R) \to   
        \end{align*}
        For $i \ge 2k-1$, this shows that $$H_i(M\setminus N; R) \to H_i(M\setminus N; 2R) \to  H_i(M\setminus N; R) $$
        is injective. Since this composition coincides with multiplication by $2$, multiplication by $2$ is injective.
        
        We claim that multiplication by $2$ coincides with the zero map: Denote by $\hat M$ its orientation cover, which also has positive sectional curvature with the pull-back metric. Denote by $\hat N$ the lift of $N$ to $\hat M$ and by $p\colon \hat M \to M$ the quotient map. $\hat N$ is still totally geodesic and therefore fulfills the assumptions of the previous case and hence $H_i(\hat M \setminus \hat N;R) = 0$. Denote by $\tau \colon H_i(M\setminus N;R) \to H_i(\hat M \setminus \hat N;R)$ the transfer homomorphism (see \cite[section 3.G]{Hatcher2001-zo}). We have $p_\ast \tau = 2 \Id$.  Since $ H_i(\hat M \setminus \hat N;R)  = 0$ for $i \ge 2k-1$, this implies $2H_i(M\setminus N;R) =0$ for $i \ge 2k -1$. Since multiplication by $2$ is injective for $i\ge 2k-1$, this implies that $H_i(M \setminus N;R) = 0$ for $i \ge 2k -1$.

        Now consider an orbifold $\OO$ with suborbifold $N$. Let $\hat \OO$ be its orientation cover and $\tau \colon \hat \OO \to \hat \OO$, such that $\OO = \hat \OO/\langle \tau \rangle$. The Connectedness Lemma implies:
        $$H_\orb^i(\hat \OO, \hat N;R) = 0 \text{ for any } 0 \le i \le n-2k +1$$
        In rational coefficients we have $H_\orb^i (\hat \OO, \hat N;\Q) = H^i(\vert \hat \OO\vert,  \vert \hat N \vert ;\Q)$. Since orientable orbifolds are rational Poincar\'e duality spaces, again $H_i(\vert \hat \OO \vert \setminus \vert \hat N\vert; \Q) = 0$ for $i \ge 2k -1$.
    The claim follows from $H_i(\vert \OO \vert \setminus \vert N\vert;\Q) = H_i(\vert \hat \OO \vert \setminus \vert \hat N\vert;\Q)^\tau$ by \cite[Theorem~III.2.4]{bredon}.
    \end{proof}
    
It turns out that in codimension one, the dual Connectedness Lemma does hold true in the Alexandrov setting:

\begin{lemma}\label{codim1}\mbox{}
    Let $A^n$ be a compact Alexandrov space with $\Curv\ge 1$ and $N$ a compact locally convex totally geodesic subspace of codimension $1$, both with empty boundary. Then each component of  $A\setminus N$ is homotopy equivalent to a point.
\end{lemma}
\begin{proof}
    Let $U$ be a connected component of $A \setminus N$. By applying the same method as in \cite[Theorem 3.3.1]{Pet07}, it follows that $f =\sin d(.,N)$ is positive  on $U$, satisfying $f''\le -f$ along unit speed geodesics. Hence the maximum is obtained at a single point. Using the gradient flow of $f$ then shows that $U$ is contractible. 
\end{proof}

\section{Counterexamples} \label{sec:counter}
\subsection{The standard version}
Let $X^n$ be an Alexandrov space with $\Curv(X) \ge 1$ and empty boundary. Then so is the suspension $\Sigma X$, with the suspension metric.  The inclusion $X \hookrightarrow\Sigma X$ is totally geodesic and of codimension $1$. Note that the inclusion map $X \hookrightarrow\Sigma X$ is null-homotopic, since $X$ contracts onto  one pole of the suspension. If the Connectedness Lemma was true, this inclusion would be $n$-connected and hence $\pi_i(X) = 0$ for $i = 0,\ldots,n-1$. Since Alexandrov spaces have the homotopy type of CW-complexes  \cite{KLP}, it follows from \cite{Mit16} that this only holds if $X$ is a homotopy sphere.
\subsection{The dual version}
We give explicit counterexamples to the dual formulation of the Connectedness Lemma for Alexandrov spaces for any coefficient ring. The examples we describe have codimension 3, but this choice is for simplicity of description only; counterexamples in codimension 2 exist as well.\\
Let $S_m$ be the permutation group of $m$ elements and $\Gamma \subset \SO(3)$  the tetrahedral group of order $12$, i.e. $$\Gamma = \langle D, P\rangle \text{ with } D = \diag(1, -1, -1) \text{ and } P(a,b,c) = (b,c,a). $$
Then $G_m \coloneqq \Gamma ^m \rtimes S_m$ acts on $V_m = (\R^3)^m$ by 
$$(\gamma_1, \ldots, \gamma_m; \sigma) \ast (x_1,\ldots, x_m) = (\gamma_1 x_{\sigma^{-1}(1)}, \ldots , \gamma_m x_{\sigma^{-1}(m)}).$$
We set \begin{align*}
    A_m &= \bb S(V_m)/G_m = \bb S^{3m-1}/G_m \text{ and }\\
     N_m &= \lbrace [x_1, \ldots, x_m] \in A_m\setvert x_i = 0 \text{ for some } i\rbrace \subseteq A_m.
\end{align*}
We claim that for $m\ge3$ the inclusion $N_m \hookrightarrow A_m$ is a counterexample to the dual version of the Connectedness Lemma for Alexandrov spaces for any coefficient ring.

Equip $A_m = \bb S(V_m)/G_m = \bb S^{3m-1}/G_m$ with the quotient metric of the round sphere and denote the quotient map by $q$. Thus $A_m$ is the underlying space of a positively curved closed orbifold without reflection points and hence a compact positively curved Alexandrov space with empty boundary of dimension $3m -1$. 
Inclusion into the first $m-1$ coordinates induces an isometry $A_{m-1} \to N_m$ and hence $\dim N_m = 3m-4$ and $\codim(N_m\subset A_m) = 3$.

We claim that  $N_m$ is totally convex: First, we show that for each nonzero  $u,v \in \R^3$ there is a $\gamma \in \Gamma$ with $\langle u,\gamma v\rangle >0$. Indeed we have $\sum_\gamma \gamma v=0$, hence, if all $\langle u, \gamma v\rangle \le 0$, they must all vanish, contradicting $u \neq 0$, since $\Gamma v$ spans $\R^3$.  For $x,y \in \bb S(V_m)$, we have $\cos(d_{A_m}(q(x),q(y))) = \max_{g \in G_m} \langle x,gy\rangle >0$ and therefore $\diam(A_m) < \pi/2$. Now, let $x,y \in q ^{-1}(N_m)$ with $d_{\bb S(V_m)}(x,y) = d_{A_m}(q(x),q(y))$. Hence $\langle x, y\rangle = \max_{g\in G_m} \langle x,gy\rangle$. Then they must have a common $0$-block: Indeed, if not, then there are $i \neq j$ with $x_i =0, \, y_i \neq 0, \, x_j \neq 0, \, y_j = 0$.

Using an element of $S_m$, we can rotate the $j$-th and the $i$-th blocks. Additionally, we choose $\gamma \in \Gamma$ such that $\langle x_j, \gamma y_i\rangle$ is maximal and positive. Denote the new element in $\bb S(V_m)$ by $y'$. By construction: $$\langle x,y'\rangle - \langle x,y \rangle = \langle x_j, \gamma y_i\rangle >0. $$
This contradicts the optimality of $y$. Hence, for some $i$ we have $x_i = y_i = 0$ and points on a minimizing geodesic between them always have vanishing $i$-th coordinate. Hence $N_m$ is totally convex.

We compute the homotopy type of $A_m \setminus N_m$: the map 
$$(t,[x_1, \ldots,x_m]) \mapsto \left[\sqrt{(1-t)\norm{x_1}^2 +\frac tm } \cdot \frac{x_1}{\norm{x_1}}, \ldots, \sqrt{(1-t)\norm{x_m}^2 +\frac tm } \cdot \frac{x_m}{\norm{x_m}} \right]$$
defines a strong deformation retract onto $$M_m := \lbrace [x_1, \ldots, x_m]\setvert \norm{x_i} = 1/\sqrt{m} \text{ for all } i = 1, \ldots m \rbrace \cong (\bb S^2/\Gamma)^m/S_m.$$ Since $\vert \bb S^2/\Gamma\vert \cong_{\homeo}\bb S^2 = \bb C \bb P^1$, we have $M_m \cong_\homeo (\bb C \bb P^1)^m/S_m $.\\
We claim $M_m \cong \bb C \bb P^m$: Let $[[a_1:b_1], \ldots, [a_m:b_m]] \in M_m$. This defines a  homogeneous polynomial $$\prod_{j = 1}^m(b_j X -a_j Y)$$ up to an overall scalar. Its coefficients define a point in $\bb C \bb P^m$. Conversely, factorizing a homogeneous polynomial in two variables over $\bb C$ recovers an unordered collection of points in $\bb C \bb P^1$ through the linear factors.  Therefore $$H_i(A_m \setminus N_m; R) = H_i (\bb C \bb P^m; R) = \begin{cases}
    R, i = 0,2, \ldots, 2m\\
    0, \text{otherwise}
\end{cases}$$ 
Hence $H_i(A_m\setminus N_m; R) = R$ for all even $2m \ge i \ge 6 = 2\codim(N_m \subset A_m)$.\\

This gives a large class of counterexamples. For this we need the following lemma. It must be known to the experts, but the authors could not find a proof of it in the literature.
\begin{lemma}\label{lem:convexity}
    Let $X$, $Y$ be compact Alexandrov spaces with curvature $\ge 1$ and $N \subset X$ totally $\pi$-convex. Then $N \ast Y \subset X \ast Y$ is totally $\pi$-convex and hence locally convex and totally geodesic.
\end{lemma}
\begin{proof}
    By \cite[Proposition I.5.10]{Haefliger}, for any Alexandrov space $A$ with $\Curv(A) \ge 1$, the geodesic segments  of length smaller than $\pi$ joining $a_1,a_2 \in A$ are in one to one correspondence with geodesic segments in the cone $C(A)$ joining $a_1,a_2 \in A \subset C(A)$. Let $c$ be a geodesic in $X \ast Y$ joining two points of $N\ast Y$ of length smaller than $\pi$. Then there exists a corresponding geodesic $\tilde c$ in $C(X \ast Y) \cong C(X) \times C(Y)$ (\cite[Proposition I.5.15]{Haefliger}), joining two points of $C(N \ast Y) = C(N) \times C(Y)$. Since $N\subset X$ is $\pi$-convex, $C(N) \subset C(X)$ is totally convex. Hence $ C(N \ast Y) \subset C(X \ast Y)$ is totally convex and we have that $\tilde c$ lies in $C(N \ast Y)$. Again by \cite[Proposition I.5.10]{Haefliger} the segment $c$ must entirely lie in $N\ast Y$.
\end{proof}
\begin{Prop}
    Let $X$ be an Alexandrov space with $\Curv(X) \ge 1$. Then also $N_m \ast X \subseteq A_m \ast X$ is a counterexample to the dual formulation of the Connectedness Lemma for any coefficient ring.
\end{Prop}
\begin{proof}
    Since $N_m \subset A_m$ is totally convex, we have by Lemma \ref{lem:convexity} that $N_m \ast X \subset A_m \ast X$ is totally $\pi$-convex and hence totally geodesic and locally convex. Furthermore, it is easy to see that $A_m \ast X \setminus N_m \ast X \simeq A_m \setminus N_m \simeq \bb C\bb P^m $. Hence the result follows as before. 
\end{proof}
\printbibliography
\end{document}